\documentclass[12pt]{article}

\usepackage[utf8]{inputenc}
\usepackage{amsmath, amsthm, amssymb, color, mathbbol}
\usepackage[shortlabels]{enumitem}
\usepackage{graphicx}
\usepackage{makecell}
\usepackage[ruled]{algorithm2e}
\usepackage{multicol}

\usepackage{booktabs}
\usepackage{mathtools}

\usepackage[table,xcdraw]{xcolor}
\usepackage{tikz}
\usepackage{array,amscd}
\usepackage{amsfonts}
\usepackage[T1]{fontenc} 
\usepackage{appendix}
\usepackage[arrow, matrix, curve]{xy}
\usepackage[square,comma,sort,numbers]{natbib}
\usepackage{booktabs}
\usepackage{siunitx}
\usepackage{multirow}
\usepackage[font=small,labelfont=bf]{caption}

\usepackage{tikz}
\usetikzlibrary{cd}
\usetikzlibrary{backgrounds}
\usetikzlibrary{decorations}
\usetikzlibrary {shapes}
\usetikzlibrary {arrows}
\usetikzlibrary {positioning}
\usetikzlibrary {calc}
\usetikzlibrary{quotes}
\usetikzlibrary{fit}

\oddsidemargin \evensidemargin
\allowdisplaybreaks[3]	

\newtheorem{definition}{Definition}[section]
\newtheorem{theorem}[definition]{Theorem}
\newtheorem*{theorem*}{Theorem}
\newtheorem{lemma}[definition]{Lemma}
\newtheorem{remark}[definition]{Remark}
\newtheorem{example}[definition]{Example}
\newtheorem*{example*}{Example}
\newtheorem{proposition}[definition]{Proposition}

\newtheorem{problem}[definition]{Problem}
\newtheorem{question}[definition]{Question}

\newcommand{\RR}{\mathbb{R}}
\newcommand{\pa}{\mathrm{pa}}
\newcommand{\an}{\mathrm{an}}
\newcommand{\An}{\mathrm{An}}

\newcommand{\lhs}{\mathrm{left}}
\newcommand{\tp}{\mathrm{top}}
\newcommand{\rhs}{\mathrm{right}}
\newcommand{\indep}{\perp \!\!\! \perp}
\DeclareMathOperator{\rank}{rank}
\newcommand{\ci}{\perp\!\!\!\perp}
\newcommand{\cd}{\,|\,}
\DeclareMathOperator{\trop}{trop}
\DeclareMathOperator{\image}{image}
\DeclareMathOperator{\trank}{trank}

\definecolor{darkgreen}{rgb}{0,0.4,0}
\definecolor{MyBlue}{rgb}{0,0.08,0.7} 
\definecolor{MyRed}{rgb}{0.85,0.08,0}

\makeatletter
\renewcommand*\env@matrix[1][\arraystretch]{%
  \edef\arraystretch{#1}%
  \hskip -\arraycolsep
  \let\@ifnextchar\new@ifnextchar
  \array{*\c@MaxMatrixCols c}}
\makeatother

\title{\textbf{Markov Equivalence of \\ Max-Linear Bayesian Networks}}

\author{Carlos Am\'{e}ndola, Ben Hollering, Seth Sullivant, Ngoc Tran}

\date{}

\begin{document}

\maketitle

\begin{abstract}
Max-linear Bayesian networks have emerged as highly applicable models for causal inference via extreme value data. However, conditional independence (CI) for max-linear Bayesian networks behaves differently than for classical Gaussian Bayesian networks. We establish the parallel between the two theories via tropicalization, and establish the surprising result that the Markov equivalence classes for max-linear Bayesian networks coincide with the ones obtained by regular CI. Our paper opens up many problems at the intersection of extreme value statistics, causal inference and tropical geometry. 
\end{abstract}

\section{Introduction}\label{sec:intro}

A \emph{max-linear Bayesian network} is a special class of a graphical model on a directed acyclic graph (DAG) used to model causal relations between extreme values \citep{gissibl2018max,kluppelberg19,engelke2020graphical}. Denoting the maximum operator $\max$ by $\vee$, its defining equation is
\begin{equation}\label{eqn:max.linear.dag.bigvee}
    X_i =\bigvee_{j=1,\dots,n}c_{ij}X_j \vee Z_i, ~ \quad ~ c_{ij}, Z_i \geq 0 \\
\end{equation}
for each $i \in [n]=\{1,\dots,n \}$, where the $Z_i$ are independent and unobserved random variables and $C$ is a matrix of coefficients supported on a DAG with $n$ nodes.
Like their classical counterparts, they are versatile and easy to interpret. They are the simplest class of models that exhibit \emph{cascading failure}, where extreme measurements  $X_j$ (rainfall, contaminant level, risk, financial return)  occurring at a large number of nodes can be traced to a few common sources $Z_i$ (storm, chemical spill, catastrophic failure, financial shock). Such cascading failures are commonly experienced in hydrology, engineering, and finance, and therefore max-linear Bayesian networks are finding many applications in these domains \citep{gissibl2018graphical,gissibl2018max, buck2020recursive, janssen2020k, kluppelberg2021estimating}. 
Most recently, \cite{ngoc.claudia.johannes} fitted max-linear Bayesian trees to data and achieved state-of-the-art results on the Hidden River problem, the current benchmark for causal discovery from extreme data \citep{asadi2015extremes}. This indicates that max-linear Bayesian networks are highly suited to model causal relations between large observed values of random variables. 

Conditional independence (CI) theory is fundamental to causal inference on Bayesian networks \citep{spirtes2000causation, pearl1995theory}.  
While max-linear Bayesian networks are special instances of Bayesian networks, they have, however, a \emph{different} CI theory \citep{CImaxlinear}. While CI statements in Bayesian networks can be found by the classical \emph{d-separation} criterion \citep{geiger1990d, geigerIdentifying90, meek1995faithful}, CI statements on max-linear Bayesian networks are given by \emph{$\ast$-separation}, a stronger form of separation (see Example \ref{ex:cassioCI}). Furthermore, in sharp contrast with the classical case, CI statements on a max-linear Bayesian network can depend on both the coefficients $C$ and the context, that is, on some (partial) realization of the conditioning set \cite[Example 1.3]{CImaxlinear}.

The natural next step is the question of Markov equivalence. Two graphs $G$ and $H$ are called \emph{Markov equivalent} if they yield the same set of CI statements under a separation criterion. Based on CI statements alone, one can only hope to recover a DAG up to Markov equivalence. Nevertheless, while Markov equivalent graphs yield the same conditional independence structures, they have very different causal structures. Understanding the Markov equivalence classes allows for the development of algorithms that infer Markov equivalence classes from data such as the PC algorithm for Bayesian networks \citep{spirtes2000causation}. 

A necessary and sufficient condition for determining Markov equivalence for classical Bayesian networks is well-known \citep{verma1990equivalence, verma1992algorithm}. Specifically, two directed acyclic graphs are Markov equivalent if and only if they have the same skeleton and same unshielded colliders (also known as $v$-structures or immoralities) \citep{AMP97}. 
In this paper we answer the natural question of determining what are the Markov equivalence classes for max-linear Bayesian networks. 

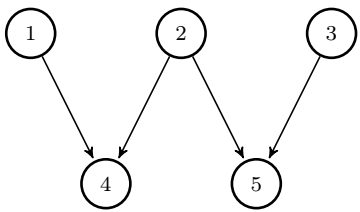
\begin{figure}
    \centering
    \begin{tikzpicture}[->,>=stealth',shorten >=1pt,auto,node distance=2cm,semithick]
        \tikzstyle{every node}=[circle,line width =1pt,font=\scriptsize,minimum height =0.65cm]
        
        \node (i1) [draw] {1};
        \node (i2) [right of = i1, draw] {2};
        \node (i3) [right of = i2, draw] {3};
        \node (i4) [below of = i1, xshift = 1cm, draw] {4};
        \node (i5) [below of = i2, xshift = 1cm, draw] {5};
        
        \path (i1) edge (i4);
        \path (i2) edge (i4);
        \path (i2) edge (i5);
        \path (i3) edge (i5);
        
    \end{tikzpicture}
    \caption{The Cassiopeia graph.}
    \label{fig:cassiopeia}
\end{figure}

Here is a motivating example. The relevant definitions will be given in Section~\ref{sec:prelim}, and more details in Example~\ref{ex:cassio2}.

\begin{example}[Cassiopeia]
\label{ex:cassioCI}
Consider the graph in Figure \ref{fig:cassiopeia}.   It holds that $1$ and $3$ are d-connected by $4$ and $5$. 
Therefore, the global Markov property is not enough to conclude that $1 \ci 3 \cd 4,5 $. 
Nevertheless, $1$ and $3$ are $*$-separated by $4$ and $5$ and hence we can conclude that for a max-linear Bayesian network supported on the Cassiopeia graph, 
the conditional independence statement 
$1 \ci 3 \cd 4,5 $ holds.
\end{example}

A priori, one would expect that the Markov equivalence class for max-linear Bayesian networks refine those seen for Gaussian Bayesian networks since it was shown in \citep{CImaxlinear} that max-linear Bayesian networks have additional valid conditional independence statements. We show in our main theorem that in fact it is not true that the Markov equivalence classes are refined, and that instead the equivalence classes of max-linear Bayesian networks coincide with the equivalence classes of Gaussian Bayesian networks.      

\begin{theorem*}[Theorem~\ref{thm:starSepEquiv}]\label{thm:first.main}
Gaussian Bayesian networks and max-linear Bayesian networks have the same Markov equivalence classes. 
\end{theorem*}
Aside from the tree case \citep{ngoc.claudia.johannes}, finding a consistent estimator for the parameters $c^\ast_{ij}$ of the max-linear Bayesian networks to data remains difficult. One primary reason is that $C^\ast$ is a matrix of max-weighted paths on an unknown DAG $G$. The set of such matrices is a non-convex piecewise-linear manifold made up of many low-dimensional cones \citep{tran2014polytropes}. Thus, estimators of $C^\ast$ are often sensitive to noise \citep{gissibl2021identifiability}. For Gaussian graphical models, conditional independence statements are governed by rank constraints on the covariance matrix (see e.g. \cite[Proposition 4.1.9]{sullivant18}). Our second main result establishes an analogue of this result for the max-linear model via tropical geometry. In particular, we show that some tropical rank constraints on the tropicalization of the covariance matrix correspond to conditional independence statements for max-linear Bayesian networks. 

\begin{theorem*}[Theorem~\ref{thm:tropRank}]\label{thm:second.main}
The conditional independence statements implied by d-separation in a max-linear Bayesian network impose tropical rank constraints that hold for every tropical covariance matrix supported on the network.
\end{theorem*}

This result opens up new directions for conditional independence testing in max-linear Bayesian networks. Namely, instead of finding a consistent estimator for $C^\ast$, which so far has proven difficult, one could opt to find a consistent estimator for the tropical covariance matrix and/or the tropical ranks of its sub-blocks. Furthermore, this theorem offers an \emph{algebraic} way to handle CI, complementing the path-based $\ast$-separation criterion of \cite{CImaxlinear}. In particular, this brings us closer in the task of identifying sufficient conditions for CI statements to hold in max-linear Bayesian networks. 

This paper is organized as follows. In Section~\ref{sec:prelim} we discuss the preliminary concepts and notation, including some basics on tropical algebra. The first main result, Theorem~\ref{thm:starSepEquiv}, and its proof are the content of Section~\ref{sec:mainresult}. We explore deeper the connection of max-linear models to tropical geometry by proving a tropical analogue of the classical trek rule in Section~\ref{sec:troptrek}. We use the tropical trek rule to obtain tropical rank constraints, culminating in our second main result, Theorem~\ref{thm:tropRank}, in Section~\ref{sec:troprank}. Finally, we consider in Section~\ref{sec:openprob} some interesting open problems and future research directions.

\section{Preliminaries}\label{sec:prelim}

\subsection{Tropical basics}

Here we recall some concepts of tropical geometry for our discussion. For an introductory exposition to this field we recommend \cite{maclagan2015introduction}.

We work in the \emph{max-times semiring} $(\mathbb{R}_\ge, \vee,\odot)$, defined by
$$a \vee b := \max(a,b), \, \, \, a \odot b := a b \, \, \, \mbox{ for } a,b \in \RR_\ge:=[0,\infty).$$

These tropical operations extend to $\RR_\ge^n$ coordinate-wise, to scalar-vector multiplication as
$$\lambda\odot x =(\lambda x_1,\ldots, \lambda x_d)\quad  \mbox{ for } \lambda\in \RR_\ge \mbox{ and } x\in \RR_\ge^n,$$ 
and to corresponding matrix product as 
\[(A\odot B)_{ij}= \bigvee_{\ell=1}^n a_{i\ell}b_{\ell j}\] 
for $A\in \RR_\ge^{m \times n}$ and $B\in \RR_\ge^{n\times p}$. In particular, this defines tropical matrix powers $A^{\odot k}$ for $k \in \mathbb{N}$ where $A^{\odot 0} = I_n$ is the identity matrix.

Analogously, we can define the \emph{tropical determinant} of a matrix $A \in \RR^{n\times n}$ as
$$\mathrm{tdet}(A) = \bigvee_{\sigma \in S_n} a_{1\sigma(1)} a_{2\sigma(2)} \dots a_{n\sigma(n)}$$
where $S_n$ denotes the symmetric group on $[n]$.

A matrix is \emph{tropically singular} if the maximum in the evaluation of the tropical determinant is attained at least twice.

\begin{definition}
The \emph{tropical rank} $\trank(M)$ of a matrix $M \in \RR^{m \times n}$ is the largest integer $r$ such that $M$ has a tropically non-singular $r \times r$ minor.
\end{definition}

\begin{example}
The matrix $M$ given by
$$\begin{pmatrix}
6 & 3 & 0 \\
0 & 8 & 4 \\
6 & 4 & 2 \\
\end{pmatrix} = \begin{pmatrix}
0 \\ 2 \\ 1 \\
\end{pmatrix} \odot  \begin{pmatrix}
0 & 4 & 2 \\
\end{pmatrix} \vee \begin{pmatrix}
3 \\ 0 \\ 3 \\
\end{pmatrix} \odot  \begin{pmatrix}
2 & 1 & 0 \\
\end{pmatrix} $$
has tropical rank 2. Indeed, the $2 \times 2$ minor 
$$A = \begin{pmatrix}
6 & 3  \\
0 & 8  \\
\end{pmatrix}$$
is tropically non-singular since 
$$\mathrm{tdet}(A) = 6 \odot 8 \, \vee \, 0 \odot 3 = 48 \vee 0 = 48$$
achieves its maximum uniquely. On the other hand, $M$ is tropically singular since
$$\mathrm{tdet}(M) = 96 \, \vee \, 84 \, \vee\,  0 \, \vee \, 0 \, \vee \,0 \, \vee \, 96 = 96 $$
attains its maximum twice, namely at the terms $6 \odot 8 \odot 2$ and $6 \odot 4 \odot 4$.
\end{example}

\begin{remark}
While tropical geometry is often introduced as min-plus or max-plus operations over $\RR \cup \{\infty \}$ or $\RR \cup \{-\infty\}$ respectively, our max-times semiring $(\mathbb{R}_\ge, \vee,\odot)$ is isomorphic to the latter by taking logarithms.
\end{remark}

\subsection{Max-Linear Bayesian Networks}

A \emph{max-linear Bayesian network} is given by a random vector $X = (X_1, \ldots X_n)$ with vertices on a directed acyclic graph $G = ([n], E)$, edge weights $c_{ij} \geq 0$, and independent positive random variables $Z_1, \ldots, Z_n$ called \emph{innovations}. The $Z_i$ have support $\RR_> = (0, \infty)$ and have atom-free distributions. Then $X$ is given by the \emph{recursive structural equations} 
\[
X_i = \bigvee_{j \in \pa(i)} c_{ij}X_j  \vee Z_i,
\]
or $X = C\odot X \vee Z$ in tropical notation. This system has solution $X= C^\ast \odot Z$, that is,
\begin{equation}\label{eqn:x.c.star}
X_i = \bigvee_{j \in \an(i) \cup {i}} c^\ast_{ij}Z_j 
\end{equation}
where $C^\ast = \bigvee_{k=0}^{n-1}C^{\odot k}$ is the \emph{Kleene star} of the matrix $C$. 
In these equations $\pa(i)$ denotes the parents
of node $i$ and $\an(i)$ denotes the ancestors of $i$.

Conditional independence in max-linear models can be quite different from conditional independence in classical Bayesian networks. For the latter, the \emph{d-separation} criterion gives a complete set of valid conditional independence statements for the model \citep{meek1995faithful}. 

We use the following standard notation (see e.g. \cite{kluppelberg19}). A \emph{path} in a DAG $G$ is a sequence of vertices $i_0, i_1, \ldots i_k$ such that $i_\ell \to i_{\ell+1}$ is an edge in $G$ or $i_{\ell+1} \to i_\ell$ is an edge in $G$ for each $\ell = 0, \ldots k$. A \emph{directed path} has edges $i_\ell \to i_{\ell+1}$ for all $\ell$. If there is a directed path from $i$ to $j$, we say that $i$ is an \emph{ancestor} of $j$ and denote by $\an(K)$ the set of all ancestors of nodes in $K$. A \emph{collider} on a path is a vertex $i_\ell$ in the path such that $i_{\ell-1} \rightarrow i_\ell \leftarrow i_{\ell+1}$. 

\begin{definition}
Two vertices $i$ and $j$ in $G$ are \emph{d-connected} given a set $K \subseteq [n] \setminus \{i, j\}$ if there is a path $\pi$ from $i$ to $j$ such that all colliders on $\pi$ are in $K \cup \an(K)$ and no non-collider on $\pi$ is in $K$. If $I, J, K \subseteq [n]$ are pairwise disjoint sets with $I$ and $J$ nonempty, then $K$ \emph{d-separates} $I$ and $J$ if no pair of nodes $i \in I$ and $j \in J$ are d-connected given $K$. 
We denote this by $I \perp_d J | K$.
\end{definition}

It was noted in \cite{kluppelberg19} that d-separation does not give all valid conditional independence statements for a max-linear Bayesian network. Recently, a new criterion named \emph{$\ast$-separation} which gives a complete set of conditional independence statements for max-linear models was given in \cite{CImaxlinear}.

\begin{definition}\label{def:starsep}
  A path $\pi$ between $i$ and $j$ in a DAG is \emph{$\ast$-connecting} given a set $K \subseteq [n] \setminus \{i,j\}$ if and only if 
  $\pi$ is d-connecting given $K$ and $\pi$ contains at most one collider.
  Two nodes $i$ and $j$ are \emph{$*$-connected} given $K$ if there is a $*$-connecting path.
  If $I, J, K \subseteq [n]$ are pairwise disjoint sets with $I$ and $J$ nonempty, then $K$ \emph{$*$-separates} $I$ and $J$ if no pair of nodes $i \in I$ and $j \in J$ are $*$-connected given $K$. 
  We denote this by $I \perp_* J | K$.
  \end{definition}

The $5$ basic shapes of $*$-connecting
paths are illustrated in Figure \ref{fig:starConnecting}.

\begin{figure}
    \centering
      \begin{tikzpicture}[->,>=stealth',shorten >=1pt,auto,node distance=2cm,semithick]
  \tikzstyle{every node}=[circle,line width =1pt,font=\scriptsize,minimum height =0.65cm]

  \node (i0) [draw] {$j$};
  \node (j0) [below of=i0,draw] {$i$} ;
  \path (i0) edge (j0) ;
  \node [below = 2.2cm of i0] {(a)};
  
\node (ksl) [right of =i0,draw]{$j'$};
\node (jsl) [below of = ksl,xshift =-0.7cm,draw] {$j$};
\node (isl) [below of = ksl,xshift =0.7cm,draw] {$i$};
\path (ksl) edge (jsl);
\path (ksl) edge (isl);

\node [below = 2.2cm of ksl] {(b)};

\node (i1) [right of =ksl, draw] {$j$};
  \node (j1) [below of=i1,xshift=0.5cm,draw,fill=gray!50] {$k$} ;
  \node (k1) [right of=i1,xshift=-1cm,draw] {$i$} ;
  \path (i1) edge (j1) ;
  \path (k1) edge (j1);  
  
  \node [below= 2.2cm  of i1,xshift =0.5cm] {(c)};
  
  \node (i2) [draw,right of = i1,xshift=1cm] {$j'$};
  \node (j2) [below of=i2,xshift=0.5cm,draw,fill=gray!50] {$k$} ;
  \node (k2) [right of=i2,xshift=-1cm,draw] {$i$} ;
  \node (i2p) [draw,left of = j2,xshift=1cm] {$j$};    
  \path (i2) edge (j2) ;
  \path (k2) edge (j2);   
  \path (i2) edge (i2p);
  
  \node [below= 2.2cm  of i2,xshift =0.5cm] {(d)};
  
  \node (i3) [draw,right of = k2] {$j'$};
  \node (j3) [below of=i3,xshift=1cm,draw,fill=gray!50] {$k$} ;
  \node (k3) [right of=i3,xshift =-0.5cm,draw] {$i'$} ;
  \node (i3p) [draw,left of = j3,xshift=0.5cm] {$j$};    
  \node (j3p) [draw,right of = j3,xshift=-0.5cm] {$i$};      
  \path (i3) edge (j3) ;
  \path (k3) edge (j3);   
  \path (i3) edge (i3p);  
  \path (k3) edge (j3p);  
  
  \node [below= 2.2cm  of i3,xshift =1cm] {(e)};
    
    \end{tikzpicture}
    \caption{Types of $\ast$-connecting paths between two nodes $i$ and $j$ with shaded nodes in $K$.}
    \label{fig:starConnecting}
\end{figure}
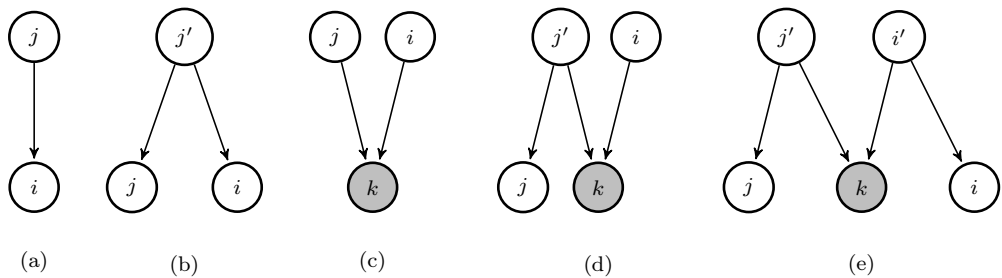

An alternate formulation of $*$-separation involves the notion of the conditional reachability DAG. This alternate formulation is useful for giving the proofs of our main result. 

\begin{definition}
  Let $G$ be a DAG and $K \subseteq [n]$. A directed path $\pi$ from $i$ to $j$ factors through $K$ if there exists a vertex $k \in \pi$ such that $k \neq i, j$ and $k \in K$. The \emph{conditional reachability DAG}, denoted $G_K^\ast$ is a graph on $[n]$ with edges given by $i \rightarrow j \in G_K^\ast$ if and only if there exists a directed path from $i \rightarrow j$ that does not factor through $K$.   
\end{definition}

\begin{example}
\label{ex:CondReachGraph}
Let $G$ be the DAG pictured on the left in Figure \ref{fig:CondReachGraph} and let $K = \{3\}$. Then the conditional reachability graph $G_K^\ast$ is the graph pictured on the right in Figure \ref{fig:CondReachGraph}. Note that the additional edge $1 \to 5$ is due to the fact that the path $1 \to 4 \to 5$ is a directed path that does not factor through $K$. On the other hand, there is no edge between 2 and 5 since the only directed path between 2 and 5 in $G$ is $2 \to 3 \to 5$ which factors through $K$.

Observe that since $1 \to 5$ is a $\ast$-connecting path in $G_K^\ast$ we have that $1 \not \perp_\ast 5 | 3$. The only path between $2$ and $5$ in $G_K^\ast$ is $2 \to 3 \to 5$ which is not $\ast$-connecting since $3 \in K$ so $2 \perp_\ast 5 | 3$. 
\end{example}

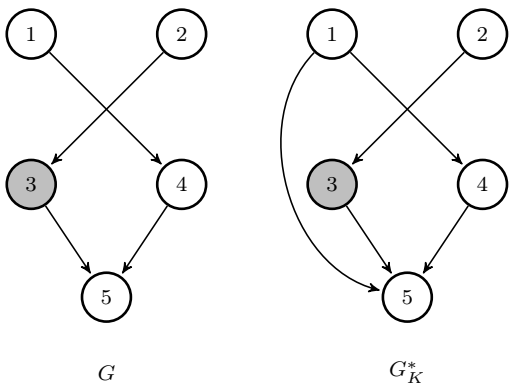
\begin{figure}
    \centering
    \begin{tikzpicture}[->,>=stealth',shorten >=1pt,auto,node distance=2cm,semithick]
    \tikzstyle{every node}=[circle,line width =1pt,font=\scriptsize,minimum height =0.65cm]
    
    \node (i1) [draw] {1};
    \node (i2) [right of = i1, draw] {2};
    \node (i3) [below of = i1, draw, fill=gray!50] {3};
    \node (i4) [below of = i2, draw] {4};
    \node (i5) [below of = i3, xshift = 1cm, yshift = .5cm, draw] {5};
    \node (l1) [below of = i5, yshift = 1cm] {$G$}; 
    
    \path (i1) edge (i4);
    \path (i2) edge (i3);
    \path (i3) edge (i5);
    \path (i4) edge (i5);
    
    \node (j1) [right of  = i2, draw] {1};
    \node (j2) [right of = j1, draw] {2};
    \node (j3) [below of = j1, draw, fill=gray!50] {3};
    \node (j4) [below of = j2, draw] {4};
    \node (j5) [below of = j3, xshift = 1cm, yshift = .5cm, draw] {5};
    \node (l2) [below of = j5, yshift = 1cm] {$G_K^\ast$};
    
    \draw [bend right = 60, ->] (j1) to (j5);
    \path (j1) edge (j4);
    \path (j2) edge (j3);
    \path (j3) edge (j5);
    \path (j4) edge (j5);
    
    \end{tikzpicture}
    \caption{A DAG $G$ and the corresponding conditional reachability graph $G_K^\ast$ for $K = \{ 3\}.$}
    \label{fig:CondReachGraph}
\end{figure}

\begin{remark}
With this definition, we can say that for $I, J, K \subseteq [n]$ pairwise disjoint sets, $I$ and $J$ are \emph{$*$-separated} 
by $K$ in $G$ if and only if there are no $*$-connecting paths from $I$ to $J$ in $G^*_K$. 
\end{remark}

Checking $*$-separation in $G^*_K$ is the analogue of checking d-separation as undirected separation in a corresponding moralized graph, see \cite[Proposition 3.25]{lauritzen1996graphical}. 

\begin{example}\label{ex:cassio2}
To provide some intuition on the \emph{`at most one collider'} condition in Definition~\ref{def:starsep}, we consider again the Cassiopeia graph from Example~\ref{ex:cassioCI} and the path between $1$ and $3$ with two colliders in $K=\{ 4,5\}$. For simplicity, let all nonzero coefficients $c_{ij}$ equal to one. Then the max-linear model states that
\[ X_1 \vee X_2 \leq  x_4 \quad \text{ and } \quad  X_2 \vee X_3 \leq  x_5 . \]
If $x_4 < x_5$, then $X_3 \leq x_5$ so that $x_5$ cannot be caused by $X_2$ but only by $X_3$ or $Z_5$. Analogously, if $x_4 \geq x_5$ then $x_4$ can only be caused by $X_1$ or $Z_4$. Finally, if $x_4 = x_5$ we must have $X_2 = x_4 = x_5$ (since the $Z_i$ are atom-free) so that both are caused by $X_2$. In any of the three cases, we have that $1 \ci 3 \cd \{4,5\}$.
\end{example}

\section{Solving Markov Equivalence}\label{sec:mainresult}
In this section we compare Markov equivalence under d-separation and $\ast$-separation and show that they give the same Markov equivalence classes.

\begin{definition}\label{def:Markoveq}
Two graphs $G$ and $H$ with vertex set $V$ are called \emph{Markov equivalent} if they yield the same set of conditional independence statements under a global Markov property, i.e., for all pairwise disjoint $I, J, K \subseteq V$,
\[ I \perp J \cd K \implies I \ci J \cd K   \]
where $\perp$ is a separation criterion. 
\end{definition}

The following theorem by \cite{verma1990equivalence,verma1992algorithm} characterizes which graphs are Markov equivalent when considering the d-separation criterion. An alternative proof can also be found in \citep{AMP97}. 

\begin{theorem}[\cite{verma1990equivalence}, Theorem 1] \label{thm:dSepMarkovEquiv}
Two directed acyclic graphs $G$ and $H$ are Markov equivalent under the d-separation criterion if and only if the following two conditions hold:
\begin{enumerate}
    \item $G$ and $H$ have the same skeleton, which is the undirected graph obtained by removing edge directions. 
    \item $G$ and $H$ have the same unshielded colliders, which are triples $i, j, k \in [n]$ which induce a subgraph of the form $i \rightarrow k \leftarrow j$.  
\end{enumerate}
\end{theorem}

We use $\sim_d$ and $\sim_\ast$ to denote Markov equivalence under d-separation and $\ast$-separation respectively. Since every $\ast$-connecting path is also d-connecting, it holds that if $I \perp_d J | K$ then $I \perp_\ast J | K$.  Consequently, one might think that the additional statements obtained from $\ast$-separation would refine the Markov equivalence classes of d-separation but this is not the case. We will show that the Markov equivalence classes are actually the same. The following lemma will be useful in the proof.

\begin{lemma}
\label{lemma:starSurgery}
Let $K \subseteq [n]$. Let $G$ and $H$ be DAGs on $[n]$ such that $G \sim_d H$ and suppose $i \rightarrow j \in G_K^\ast$. Then one of the following holds:
\begin{enumerate}
    \item $i \rightarrow j \in H_K^\ast$,
    \item $j \rightarrow i \in H_K^\ast$,
    \item $i$ and $j$ have a common parent $\ell$ in $H_K^\ast$, i.e., there exists $l \in [n]$ such that $\ell \rightarrow i, \ell \rightarrow j \in H_K^\ast$.  
\end{enumerate}
\end{lemma}
\begin{proof}
The existence of the edge  $i \rightarrow j \in G_K^\ast$ implies that there is a directed path $\pi_G$ from $i$ to $j$ in $G$ that does not factor through $K$. Since it holds that 
$G \sim_d H$ hence $G$ and $H$ have the same skeleton. This means there exists a path $\pi_H$ between 
$i$ and $j$ in $H$ whose edge directions we now consider. 

First we suppose that $\pi_H$ has no colliders on it. Then $\pi_H$ must be a directed path from $i$ to $j$, a directed path from $j$ to $i$, or has exactly one source $\ell \neq i, j$ on it. If $\pi_H$ is a directed path from $i$ to $j$ then (1) holds since there is no vertex in $K$ on $\pi$ and similarly if $\pi_H$ is a directed path from $j$ to $i$ then (2) holds. If there is a source $\ell \neq i, j$ then we get that (3) holds since there must be a directed paths from $\ell$ to both $i$ and $j$ and no vertex on these paths can be in $K$. 

Now we suppose that there is at least one collider on $\pi_H$. We will show a different path exists by looking at the colliders on the path. Suppose $v$ is a collider on $\pi_H$ and let $u \rightarrow v \leftarrow w$ be path $\pi_H$ locally around $v$. The local path around $v$ in $\pi_G$ must have the form $u \rightarrow v \rightarrow w$ since $\pi_G$ is the path that induces the edge $i \rightarrow j \in G_K^\ast$. Since we have that $G \sim_d H$ they must have the same unshielded colliders by Theorem \ref{thm:dSepMarkovEquiv} and since the triple $(u, v, w)$ is not an unshielded collider in $G$, it cannot be an unshielded collider in $H$. Combining this with the fact that locally around $v$, $\pi_H$ has the form $u \rightarrow v \leftarrow w$, there must either exist an edge $u \rightarrow w$ or $u \leftarrow w$ in $H$. So we can create a new path, $\pi_H'$, in $H$ by replacing $u \rightarrow v \leftarrow w$ with the edge between $u$ and $w$. The path $\pi_H'$ is one edge shorter than $\pi_H$ and has one less collider or one of the vertices $u, w$ has become a collider. We can inductively apply the same argument to $\pi_H'$ though until we obtain a path between $i$ and $j$ with no colliders or we end up with a direct edge between $i$ and $j$. If we have a direct edge between $i$ and $j$ then either (1) or (2) holds and if we have a path with no colliders than the result holds by the previous paragraph.  
\end{proof}

We are now ready to prove our promised first main result.

\begin{theorem}
\label{thm:starSepEquiv}
d-separation and $*$-separation induce the same Markov equivalence classes on a DAG $G$.
\end{theorem}

\begin{proof}
Let $G \sim_d H$. We will show that for any $K \subseteq [n]$, if there exists a $\ast$-connecting path $\pi_G$ between $i$ and $j$ in $G$ then there is a $\ast$-connecting path $\pi_H$ between $i$ and $j$ in $H$. This implies that $G$ and $H$ have the same $\ast$-separations and hence are Markov equivalent with respect to $\ast$-separation. So fix $K$ and let $\pi_G$ be a $\ast$-connecting path between $i$ and $j$ in $G$ conditioned on $K$. We now argue that the desired path $\pi_H$ exists based on which of the five possible forms displayed in Figure \ref{fig:starConnecting} $\pi_G$ may take. For each of the possible $\ast$-connecting paths $\pi_G$, we apply Lemma \ref{lemma:starSurgery} to each edge in $\pi_G$ and analyze the possible resulting graphs. Throughout the rest of the proof, we denote by (1), (2), and (3) the three outcomes that we can get by applying Lemma \ref{lemma:starSurgery} to an edge in $\pi_G$. 

\begin{enumerate}[label=(\alph*), wide, labelwidth=!, labelindent=0pt]
\item Suppose $\pi_G$ has the form of path (a) in Figure \ref{fig:starConnecting}. Then  Lemma \ref{lemma:starSurgery} gives three possibilities which are all $\ast$-connecting paths between $i$ and $j$. 

\item Suppose $\pi_G$ has the form of path (b) in Figure \ref{fig:starConnecting}. We analyze these cases up to the symmetry obtained by interchanging $i$ and $j$.
\begin{itemize}
    \item If (1) holds for both edges then the path $j \leftarrow j' \rightarrow i$ is also in $H_K^\ast$.
    \item If (1) holds for either edge with (2) holding for the other then we have a directed path between $i$ and $j$ in $H_K^\ast$ which implies the existence of a directed path between $i$ and $j$ in $H$ that does not factor through $K$ hence we $i \rightarrow j \in H_K^\ast$ which is $\ast$-connecting as well.
    \item If (1) holds for $j' \rightarrow j$ and (3) holds for $j' \rightarrow i$ then there exists an $\ell$ such that $j' \leftarrow \ell \rightarrow i \in H_K^\ast$. The existence of these two paths implies that $j \leftarrow \ell \rightarrow i \in H_K^\ast$. 
    \item If (2) holds for both edges, (3) holds for both edges, or (2) holds for one edge and (3) for the other then $j'$ would be a collider in $H_K^\ast$. If such a collider exists though then we are guaranteed a path directly from $i$ to $j$ in $H$ by the proof of Lemma \ref{lemma:starSurgery} which means we have a path of type (a) in $H_K^\ast$. 
\end{itemize}

\item Suppose $\pi_G$ has the form of path (c) in Figure \ref{fig:starConnecting}. We again use the symmetry between $i$ and $j$ to reduce the number of cases.  
    \begin{itemize}
        \item If (1) holds for both edges then the path is unchanged and hence $\ast$-connecting.
        \item Suppose (1) holds for the edge $i \rightarrow k$ and (2) holds for the edge $j \rightarrow k$.
        Then locally around $k$, $\pi_G$ has the form $u \rightarrow k \leftarrow v$ but in $H$ this path must be of the form $u \rightarrow k \rightarrow v$ which means that the triple $(u, k, v)$ is no longer an unshielded collider in $H$ though it was in $G$. This contradicts the assumption that $G \sim_d H$ though so this scenario is not possible. The same argument applies to the case where (2) holds for both edges and the case where (2) and (3) hold for the edges. 
        \item Suppose (1) holds for the edge $i \rightarrow k$ and (3) holds for the edge $j \rightarrow k$. Then there exists a common parent $j'$ of $j$ and $k$ which implies the existence of a path of type (d) between the vertices $i, j, j',$ and $k$. 
        \item Suppose (3) holds for both edges. Then there exists a common parent $i'$ of $i$ and $k$ as well as a common parent $j'$ of $j$ and $k$. This means that there is a path of type (e) between these vertices which is $\ast$-connecting. 
    \end{itemize}
    
    \item Suppose $\pi_G$ has the form of path (d) in Figure \ref{fig:starConnecting}. We consider the different cases that can arise based on the cases we had when $\pi_G$ had the form (c) and analyze how adding the edge between $j$ and $j'$ affects these cases. 
    \begin{itemize}
        \item Suppose (1) holds for the edge $i \rightarrow k$ and (2) holds for the edge $j' \rightarrow k$. Regardless of the status of the edge $j' \rightarrow j$, we know from the previous case that the graph $H$ is missing an unshielded collider which contradicts $G \sim_d H$. Just as in the previous case, the same argument applies to the case where (2) holds for both edges and the case where (2) and (3) hold for the edges.  
        \item Now suppose that any of the other cases hold. Then we know there is a $\ast$ connecting path between $j'$ and $i$ of type (c), (d), or (e). Furthermore, there is a $\ast$-connecting path between $j$ and $j'$ of type (a) or (b). If the path between $j'$ and $i$ is of type (c) then taking the union of this path with any of the configurations of the path between $j$ and $j'$ gives a $\ast$-connecting path between $i$ and $j$ of type (c), (d), or (e) still. If the path between $j'$ and $i$ is of type (e) and the path between $j$ and $j'$ is of type (b) then $j'$ will be a collider in the union of these paths. This path has the form $j \leftarrow m \rightarrow j' \leftarrow \ell \rightarrow k \leftarrow \ell' \rightarrow i$. Again though by the proof of Lemma \ref{lemma:starSurgery}, we know that the parents of $j'$ must be moral since $j'$ is not a collider on $\pi_G$ and $G \sim_d H$. This means that we have an edge $m \rightarrow \ell$ and thus the path $j \leftarrow m \rightarrow k \leftarrow \ell \rightarrow i$ which is a path of type (d) between $i$ and $j$. The remaining configurations follow in the exact same way. 
    \end{itemize}
    
     \item Suppose $\pi_G$ has the form of path (e) in Figure \ref{fig:starConnecting}. This case follows in a very similar way to the type (d) case. We can again rule out all of the cases where $k$ is no longer a collider. The remaining cases then follow from the same argument used previously. We know that $j$ and $i'$ will have a $\ast$-connecting path between them of type (c), (d), or (e) by the previous case. We also know from the previous case that taking the union of this path with any of the possible configurations of the path between $i'$ and $i$ also gives a $\ast$-connecting path. 
\end{enumerate}

This completes the proof since the desired path exists in every case.
\end{proof}

\begin{remark}
The equality of equivalence classes in Theorem~\ref{thm:starSepEquiv}  holds for generic coefficient matrices $C$ supported on $G$. For special choices of $C$, there may be more valid CI statements.
\end{remark}

The following example illustrates Theorem \ref{thm:starSepEquiv}. 

\begin{example}
\label{ex:starSepEquiv}
Let $G$ be the graph pictured in Figure \ref{fig:starSepEquiv}. Let $H$ be obtained from $G$ by reversing the edge $3 \to 6$ and let $F$ be obtained by reversing the edge $6 \to 7$ in $H$. Then the graphs $G, H$ and $F$ form a Markov equivalence class under d-separation since each of the reversals of edges does not change the unshielded colliders but altering any other edges would. It is also straightforward to check that any additional conditional independence statements that come from applying the $\ast$-separation criterion to $G$ are also valid for $H$ and $F$. For example, suppose $K = \{4, 5\}$ and observe that in all three conditional reachability graphs there is a unique path between 1 and 6 in which there are always two colliders and thus cannot be $\ast$-connecting. So the additional statements that $\ast$-separation gives do not refine this Markov equivalence class for d-separation and thus we see this is also a Markov equivalence class with respect to $\ast$-separation.
\end{example}

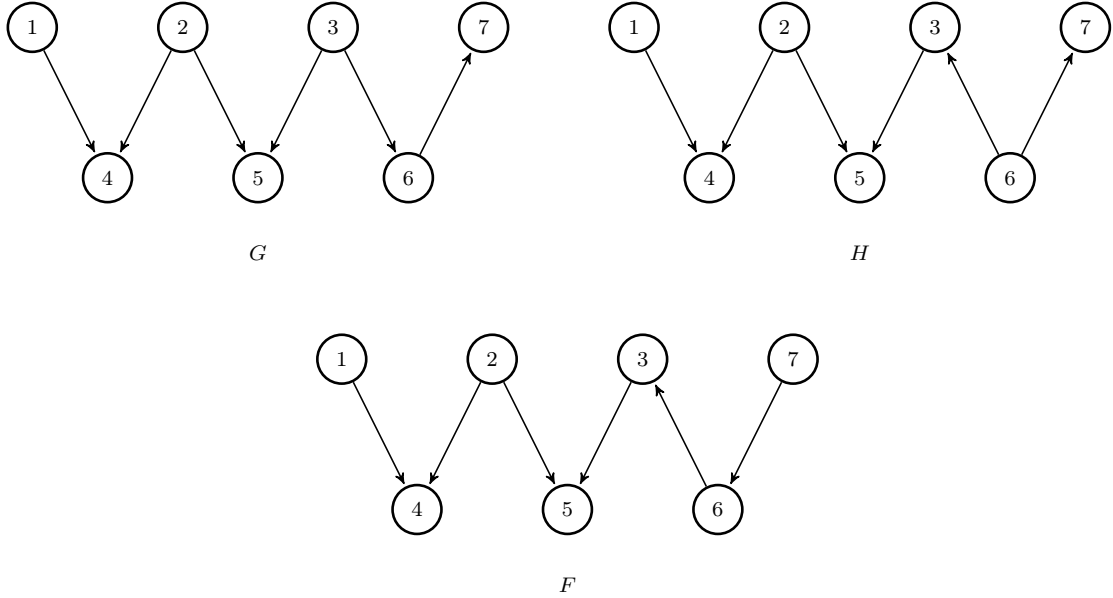
\begin{figure}[t]
\begin{center}
    \begin{tikzpicture}[->,>=stealth',shorten >=1pt,auto,node distance=2cm,semithick]
        \tikzstyle{every node}=[circle,line width =1pt,font=\scriptsize,minimum height =0.65cm]
        
        \node (i1) [draw] {1};
        \node (i2) [right of = i1, draw] {2};
        \node (i3) [right of = i2, draw] {3};
        \node (i4) [below of = i1, xshift = 1cm, draw] {4};
        \node (i5) [below of = i2, xshift = 1cm, draw] {5};
        \node (i6) [below of = i3, xshift = 1cm, draw] {6};
        \node (i7) [right of = i3, draw] {7};
        \node (l1) [below of = i5, yshift = 1cm] {$G$};
        
        \path (i1) edge (i4);
        \path (i2) edge (i4);
        \path (i2) edge (i5);
        \path (i3) edge (i5);
        \path (i3) edge (i6);
        \path (i6) edge (i7);
        
        \node (j1) [right of = i7, draw] {1};
        \node (j2) [right of = j1, draw] {2};
        \node (j3) [right of = j2, draw] {3};
        \node (j4) [below of = j1, xshift = 1cm, draw] {4};
        \node (j5) [below of = j2, xshift = 1cm, draw] {5};
        \node (j6) [below of = j3, xshift = 1cm, draw] {6};
        \node (j7) [right of = j3, draw] {7};
        \node (l2) [below of = j5, yshift = 1cm] {$H$};
        
        \path (j1) edge (j4);
        \path (j2) edge (j4);
        \path (j2) edge (j5);
        \path (j3) edge (j5);
        \path (j6) edge (j3);
        \path (j6) edge (j7);
        
        \node (k1) [below right of = i5, xshift = -.3cm, yshift = -1cm, draw] {1};
        \node (k2) [right of = k1, draw] {2};
        \node (k3) [right of = k2, draw] {3};
        \node (k4) [below of = k1, xshift = 1cm, draw] {4};
        \node (k5) [below of = k2, xshift = 1cm, draw] {5};
        \node (k6) [below of = k3, xshift = 1cm, draw] {6};
        \node (k7) [right of = k3, draw] {7};
        \node (l3) [below of = k5, yshift = 1cm] {$F$};
        
        \path (k1) edge (k4);
        \path (k2) edge (k4);
        \path (k2) edge (k5);
        \path (k3) edge (k5);
        \path (k6) edge (k3);
        \path (k7) edge (k6);
        
    \end{tikzpicture}
    \caption{The graphs $G$, $H$, and $F$ used in Example \ref{ex:starSepEquiv}. These three graphs form a Markov equivalence class with respect to d- and $\ast$-separation. }
    \label{fig:starSepEquiv}
    \end{center}
\end{figure}

Note that as a corollary of Theorem \ref{thm:starSepEquiv}, the problem of counting the number of Markov equivalence classes for max-linear Bayesian networks is equivalent to the classical one, as studied in \cite{gillispie2001enumerating,radhakrishnan2017counting,radhakrishnan2018counting}.

\section{Tropical Trek Rule}\label{sec:troptrek}
In this section we define a new matrix, $\Sigma^\mathrm{trop}$, associated to a max-linear model that is a natural analogue of the covariance matrix for Gaussian distributions. We then show that this matrix can be obtained by tropicalizing the trek rule of \cite{STD10} which is used to parameterize the covariance matrix of a directed Gaussian graphical model. This immediately implies tropical rank constraints on $\Sigma^\mathrm{trop}$ which correspond to conditional independence statements for the model. 

We begin with some background on directed Gaussian graphical models and the trek rule. For additional background we refer the reader to \cite[Chapter 13]{sullivant18}. Recall that a random vector $X$ is distributed according to the directed Gaussian graphical model for a graph, $G$, if it satisfies the recursive structural equations
$X = C X + Z$ where $C$ is the weighted adjacency matrix of $G$ and $Z$ is a Gaussian random vector with diagonal covariance matrix $\Omega$. The matrices $C$ and $\Omega$ are the parameters of the model. The recursive structural equations have solution
$X = (Id-C)^{-1} Z$ where $(Id - C)^{-1}$ plays the role that $C^\ast$ plays in the max-linear model. The covariance matrix $\Sigma$ admits the factorization $\Sigma = (Id - C)^{-1} \Omega (Id-C)^{-T}$ and this factorization can be used to interpret the entries of $\Sigma$ combinatorially. 

\begin{definition}
\label{def:trek}
A \emph{trek} $\tau$ from $i$ to $j$ is an alternating sequence of nodes and edges of the form
\[ i\leftarrow i_{l}\leftarrow \dots \leftarrow  i_1\xleftarrow[]{\hspace{0.5cm}} i_0\xrightarrow[]{\hspace{0.5cm}}  j_1\to \dots \to j_{r}\to j. \]
(a trek takes you up and down a `mountain'). The \emph{top} of the trek is $\tp(\tau) = i_0$, the \emph{left-hand side} of the trek, $\lhs(\tau) = \{i_0,i_1,\dots,i_{l},i\}$ and the \emph{right-hand side} of the trek is $\rhs(\tau) = \{i_0,j_1,\dots,j_{r},j\}$. We also allow \emph{trivial} treks with a single node $i$ that have $\lhs(\tau) = \rhs(\tau) = \{ i \}$.
\end{definition}

For a trek $\tau$ with top $i_0$ in a DAG with edge weights given by $C$ we can naturally define a trek monomial:
\begin{equation}
\label{def:trekmon}
    \tau(C,\Omega) = \omega_{i_0 i_0} \prod_{k\to l\in \tau} c_{lk}.
\end{equation}

\begin{proposition}[Trek Rule]
(see e.g.~ \cite{STD10})
\label{prop:trekrule}
Let $X$ be distributed according to the directed Gaussian graphical model on $G$ with parameters $C$ and $\Omega$. Then the covariance matrix, $\Sigma$, of $X$ satisfies
\begin{equation}\label{eq:trekrule}
    \Sigma_{ij} = \sum_{\tau\in\mathcal{T}(i,j)}
    \tau(C,\Omega), \qquad i,j\in [n].
\end{equation}
where $\mathcal{T}(i,j)$ denotes the set of all treks from $i$ to $j$.
\end{proposition}

Since the structural equations of a max-linear model are given by tropicalizing the structural equations of a directed Gaussian graphical model, it is natural to consider the tropicalization of the above trek rule. If the random variable $X$ is distributed according to the max-linear model on $G$ with coefficient matrix $C$ then we call $\Sigma^{\mathrm{trop}} = C^\ast \odot (C^\ast)^T$ the \emph{tropical covariance matrix} for $X$ where the matrix multiplication is in max-times arithmetic. This definition is motivated by the factorization that the covariance matrix of a directed Gaussian graphical model admits, as shown above. 

\begin{theorem}[Tropical Trek Rule]
Let $G = ([n], E)$ be a DAG and $\Sigma^{\mathrm{trop}} = C^\ast \odot (C^\ast)^T$ for a coefficient matrix $C$ supported on $G$. Then
\begin{equation}\label{eq:troptrek}
(\Sigma^{\mathrm{trop}})_{ij} = \bigvee_{\tau \in \mathcal{T}(i,j)} \tau(C, Id).
\end{equation}
\end{theorem}

\begin{proof}
First note that by construction $c_{ij}^\ast$ is the maximum weight of any path from $j$ to $i$ which means
\[
c_{ij}^\ast = \bigvee_{\pi \in \mathcal{P}(j,i)} \prod_{m \to \ell \in \pi} c_{\ell m}. 
\]
This means that $(C^\ast)_{ij}^T = c_{ji}^\ast$ is the maximum weight of any path from $i$ to $j$. Thus we get that
\begin{align*}
(C^\ast & \odot  (C^\ast)^T)_{ij} = \bigvee_{k} c_{ik}^\ast c_{jk}^\ast 
 \\ &= \bigvee_{k} 
 \left( \bigvee_{\pi \in \mathcal{P}(k,i)} \prod_{m \to \ell \in \pi} c_{\ell m} \right)
\left( \bigvee_{\pi \in \mathcal{P}(k,j)} \prod_{m \to \ell \in \pi} c_{\ell m} \right)
\end{align*}
Note that the last expression corresponds exactly to the trek monomial of the max-weighted trek between $i$ and $j$ which gives the desired result.
\end{proof}

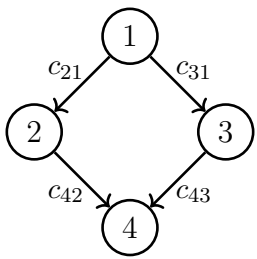
\begin{figure}
    \centering
    \begin{tikzpicture}
        \begin{scope}[->,every node/.style={circle,draw},line width=1pt, node distance=1.8cm]
        \node (1) {$1$};
        \node (2) [below left of=1] {$2$};
        \foreach \from/\to in {1/2}
        \draw (\from) -- (\to);
        \path[every node/.style={font=\sffamily\small}]
        (1) -- (2) node [near start, left] {$c_{21}$};
        \node (3) [below right of=1] {$3$};
        \foreach \from/\to in {1/3}
        \draw (\from) -- (\to);
        \path[every node/.style={font=\small}]
        (1) -- (3) node [near start, right] {$c_{31}$};
        \node (4) [below right of=2] {$4$};
        \foreach \from/\to in {2/4,3/4}
        \draw (\from) -- (\to);
        \path[every node/.style={font=\small}]
        (2) -- (4) node [near end, left] {$c_{42}$};
        \path[every node/.style={font=\small}]
        (3) -- (4) node [near end, right] {$c_{43}$};
        \end{scope}
    \end{tikzpicture}  
    \caption{The Diamond graph $G$ with its edges labeled by coefficients.}
    \label{fig:diamond}
\end{figure}

\begin{example}
Let $G$ be the Diamond graph which is pictured in Figure \ref{fig:diamond}. The tropical trek rule can be used to compute the entries of the tropical covariance matrix corresponding to the max-linear model on $G$.  

For example, there are three treks from $2$ to $4$: 
$$2 \rightarrow 4, \quad 2 \leftarrow 1 \rightarrow 2 \rightarrow 4 \, \, \text{ and } \, \, 2 \leftarrow 1 \rightarrow 3 \rightarrow 4 .$$ Then, according to expression \eqref{eq:troptrek}:
\begin{equation}\label{eq:exampletrekrule}
    \Sigma^\mathrm{trop}_{24} = c_{42} \vee c_{21}^2 c_{42} \vee c_{21}c_{31}c_{43}.
\end{equation}
\end{example}

\section{Tropical Rank Constraints}\label{sec:troprank}

The tropical trek rule allows us to easily show that conditional independence statements that come from d-separation correspond to tropical rank constraints on $\Sigma^\mathrm{trop}$. 

We first recall the following proposition which is the analogous result for Gaussians. 
 
\begin{proposition}
\label{prop:GaussCond}
Let $X$ be a multivariate Gaussian with covariance matrix $\Sigma$ and $I, J, K \subseteq [n]$ be disjoint sets. Then the conditional independence statement $X_I \indep X_J | X_K$ holds if and only if $\rank(\Sigma_{I \cup K, J \cup K}) = \#K$. 
\end{proposition}

The trek rule can also be thought of as a map that parameterizes the set of $\Sigma$ that can be produced by the directed Gaussian graphical model on a DAG, $G$. Let $G = ([n], E)$ be a DAG and let
\begin{align*}
    \phi_G : \mathbb{R}^{E} \times \mathbb{R}^{n} &\to  \mathbb{R}^{\binom{n+1}{2}}
\end{align*}
be defined by $$\phi_G(C, \Omega)_{ij} = \sum_{\tau\in\mathcal{T}(i,j)}\tau(C, \Omega).$$ The image of $\phi_G$ is exactly the parameterized Gaussian graphical model associated to $G$. From Section 5 of \cite{ST08} we have that
\[
\image(\trop(\phi_G)) \subseteq \trop(\image(\phi_G)). 
\]
The following two results are an immediate consequence of this containment of tropical varieties. 

\begin{theorem}
\label{thm:tropRank}
Let $G$ be a DAG and $\Sigma^\mathrm{trop}$ be supported on $G$. If $K$ d-separates $I$ and $J$ in the DAG $G$ then $\trank(\Sigma_{I\cup K,J \cup K}^\mathrm{trop}) = \#K$. 
\end{theorem}
\begin{proof}
Proposition \ref{prop:GaussCond} implies that if $\Sigma \in \image(\phi_G)$
every $(\#K + 1) \times (\#K+1)$ minor of $\Sigma_{I,J}$ vanishes. Since we have that $\Sigma^\mathrm{trop} \in \trop(\image(\phi_G))$, the tropicalization of any polynomial that vanishes on $\image(\phi_G)$ will vanish on $\Sigma^\mathrm{trop}$. So all of the $(\#K + 1) \times (\#K+1)$ tropical minors of $\Sigma^\mathrm{trop}$ vanish and the result follows. 
\end{proof}

\begin{example}
\label{ex:trankDrop}
Consider the Diamond graph $G$ pictured in Figure \ref{fig:diamond} which has $\Sigma^\mathrm{trop}$ equal to
\[
 \begin{pmatrix}
1           &  c^\ast_{21}             & c^\ast_{31}               & c^\ast_{41}\\
c^\ast_{21} & (c^\ast_{21})^2\vee 1    & c^\ast_{21}c^\ast_{31}    & c^\ast_{21}c^\ast_{41}\vee c^\ast_{42}\\
c^\ast_{31} & c^\ast_{21}c^\ast_{31}             & (c^\ast_{31})^2\vee 1            &  c^\ast_{31}c^\ast_{41}\vee c^\ast_{43}\\
c^\ast_{41}     & c^\ast_{21}c^\ast_{41}\vee c^\ast_{42}  & c^\ast_{31}c^\ast_{41}\vee c^\ast_{43}   & (c^\ast_{41})^2 \vee (c^\ast_{42})^2\vee (c^\ast_{43})^2 \vee 1\\
\end{pmatrix}
\]

Note that the entry $\Sigma_{24}^\mathrm{trop}$ coincides with that computed in Equation \eqref{eq:exampletrekrule} since
\begin{align*}
    \Sigma^\mathrm{trop}_{24} &= c_{42} \vee c_{21}^2 c_{42} \vee c_{21}c_{31}c_{43} \\
     &= c_{42}^\ast \vee c_{21}^\ast( c_{21}c_{42} \vee c_{31}c_{43}) \\
     &= c_{42}^\ast \vee c_{21}^\ast c_{41}^\ast.
\end{align*}

Observe that $K = \{1\}$ d-separates $I = \{2\}$ and $J = \{3\}$ in $G$ so the tropical rank of the submatrix  $\Sigma^\mathrm{trop}_{I \cup K, J \cup K} = \Sigma^\mathrm{trop}_{\{1,2\}, \{1,3\}}$ should be $\#K = 1$. More explicitly, the submatrix is
\[
\Sigma^\mathrm{trop}_{\{1,2\}, \{1,3\}} = 
\begin{pmatrix}
1 & c^\ast_{31}\\
c^\ast_{21} &   c^\ast_{21}c^\ast_{31}\\
\end{pmatrix}
\]
and is not zero so it has tropical rank at least one. To show that it is rank is 1, we compute the tropical determinant which is
\[
\det(\Sigma^\mathrm{trop}_{\{1,2\}, \{1,3\}}) =  c^\ast_{21}c^\ast_{31} \vee c^\ast_{31} c^\ast_{21} 
\]
Since this determinant is tropically singular, we have that $\trank(\Sigma^\mathrm{trop}_{\{1,2\}, \{1,3\}}) = \#K = 1 .$ 
\end{example}

While conditional independence statements that come from d-separation give tropical rank constraints on $\Sigma^\mathrm{trop}$, the same is not necessarily true for those which come from $\ast$-separation. The following example illustrates this. 

\begin{example}
\label{ex:cassioTrank}
Let $G$ be the Cassiopeia graph pictured in Figure \ref{fig:cassiopeia} and recall that $1 \perp_\ast 3 | \{4, 5\}$ in $G$. The corresponding block of $\Sigma^\mathrm{trop}$ is
\[
\Sigma_{\{1, 4, 5\}, \{3, 4, 5\}}^\mathrm{trop} =
\begin{pmatrix}
0 & c_{41} & 0 \\
0 & c_{41}^2 \vee c_{42}^2 \vee 1 & c_{42}c_{52} \\
c_{53} & c_{42}c_{52} & c_{52}^2 \vee c_{53}^2  1
\end{pmatrix}.
\]
Observe that the tropical determinant of this submatrix is
\[
\det(\Sigma_{\{1, 4, 5\}, \{3, 4, 5\}}^\mathrm{trop}) = 0 \vee 0 \vee 0 \vee c_{41}c_{42}c_{52}c_{53} \vee 0 \vee 0
\]
which is not tropically singular for any choice of $C$. This means for every $C$ supported on $G$, the tropical rank of this submatrix is 3 so this conditional independence statement does not correspond to a drop in tropical rank. 
\end{example}

It is worth mentioning that, in general, computing the tropical rank is \emph{NP-Hard} \citep{shitov2014complexity}. However, there exist recent approximation algorithms \citep{karaev2019algorithms}. 

A natural question is if Theorem~ \ref{thm:tropRank} can be used for structure learning of a graph. A first obstacle is access to the tropical covariance matrix $\Sigma^{trop}$. Unlike the classical case, there is no known estimator for this matrix from data (see Question~\ref{q:estimator}). 

Nevertheless, we could assume oracle access to $\Sigma^{trop}$ and apply the PC algorithm to try recover a max-linear graph. However, the PC algorithm can fail because the distribution is in general not \emph{faithful}. This means that there exist valid CI statements that do not follow from $d$-separation. We illustrate with an example.

\begin{example}
\label{ex:PCfail}
Consider the Diamond graph from Figure~\ref{fig:diamond}, and assume that the matrix $C$ satisfies $c_{42}c_{21} < c_{31}c_{43}$ with $c_{31}>1$. Then we have that
\[
\Sigma^\mathrm{trop}_{\{1,3\}, \{3,4\}} = 
\begin{pmatrix}
c^\ast_{31} & c^\ast_{41}\\
(c^\ast_{31})^2\vee 1 &   c^\ast_{31}c^\ast_{41} \vee c^\ast_{43}\\
\end{pmatrix}
\]
has tropical rank 1 because
$$c^\ast_{31}(c^\ast_{31}c^\ast_{41} \vee c^\ast_{43}) =  c^\ast_{41} ((c^\ast_{31})^2\vee 1) = (c^\ast_{31})^2 c^\ast_{41}.$$ While it is true that $1 \ci 4 \cd 3$ in this scenario, this CI statement cannot be concluded from d-separation, since it is not true that $K=\{ 3\}$ d-separates $I=\{1\}$ and $J=\{4\}$ in the Diamond graph.
\end{example}

\section{Open Problems}\label{sec:openprob}
In this section we describe some open problems surrounding max-linear models with a particular emphasis on conditional independence. 

Our original inspiration for considering $\Sigma^\mathrm{trop}$ was its similarity to the tail-dependence matrix $\chi$ defined in \citep{sibuya1960bivariate}. \emph{Conditional tail dependence} is the extreme value analogue of correlation. \cite{gissibl2018max} showed that if $X$ is distributed according to a max-linear model on a DAG $G$ with Fr\'echet$(\alpha)$ innovations $Z_i$ then the tail dependence between $X_i$ and $X_j$ can be computed in the following way. First define the normalized coefficient matrix $\overline{C}$ with entries
\[
\overline{c}_{ij} = \frac{(c_{ij}^\ast)^\alpha}{\sum_{k \in \An(j)} (c_{kj}^\ast)^\alpha}
\]
then the tail dependence between $X_i$ and $X_j$ is 
\[
\chi(i,j) = \sum_{k \in \An(i) \cap \An(j)} \overline{c}_{ki} \wedge \overline{c}_{kj}.
\]
Since tail dependence is a popular measure of dependence in extreme value theory and the matrix $\chi$ can be estimated directly from data, it would be interesting to determine if something analogous to Theorem \ref{thm:tropRank} holds for $\chi$. If a relationship like this could be found, then more tools from algebraic geometry and tropical geometry could be used to study max-linear models just as algebraic geometry has been used to study Gaussian Bayesian networks. This motivates the following problem. 

\begin{problem}
Determine if conditional independence statements that hold for the max-linear model $X$ correspond to an algebraic or tropical algebraic constraint on the tail dependence matrix $\chi$. 
\end{problem}
It would be interesting to determine other information that $\Sigma^\mathrm{trop}$ encodes. We have shown that it satisfies tropical rank constraints similar to those for Gaussians but it would be more helpful if $\Sigma^\mathrm{trop}$ had a direct interpretation in terms of the underlying max-linear model or could be determined from data like $\chi$ (without having to estimate $C^\ast$). This leads us to the following question. 

\begin{question}\label{q:estimator}
Is there a consistent estimator for $\Sigma^\mathrm{trop}$?
\end{question}

Lastly, we note that our rank constraints on $\Sigma^\mathrm{trop}$ only correspond to conditional independence statements that come from d-separation. It would be interesting to determine if the conditional independence statements that come $\ast$-separation can also be interpreted as an algebraic constraint on $\Sigma^\mathrm{trop}$ or a related matrix such as $\chi$.

\begin{problem}
\label{prob:starSepAlg}
Suppose $K$ $\ast$-separates $I$ and $J$ in the DAG $G$ so $X_I \indep X_J | X_K$ for $X$ distributed according to a max-linear model on $G$.  Determine if this conditional independence statement corresponds to an algebraic or tropical algebraic constraint on the matrix $\Sigma^\mathrm{trop}$. 
\end{problem}

In fact, in Example~\ref{ex:PCfail} we see that despite the CI statement $1 \ci 4 \cd 3$ not being implied by d-separation, and only by $*$-separation, a tropical rank constraint still holds. However, we see in Example \ref{ex:cassioTrank} there is a $\ast$-separation statement that does not correspond to a drop in tropical rank. This suggests a complex relationship between tropical rank and $\ast$-separation.

Finally, it may also be interesting to consider an analogue of Problem \ref{prob:starSepAlg} when the coefficient matrix $C$ is fixed. This means additional $\ast$-separation statements might hold \citep{CImaxlinear}.

\subsubsection*{Acknowledgements}
  CA was partially supported by the Deutsche Forschungsgemeinschaft (DFG) in the context of the Emmy Noether junior research group KR 4512/1-1. 
   BH and SS were partially supported by the US NSF (DMS 1615660). NT is supported by the US NSF (DMS 2113468).
   We are grateful to anonymous reviewers for their constructive feedback on the paper.

\bibliography{ref}
 
\bigskip 

\bigskip

\small {\bf Authors' addresses:}

\bigskip 

\noindent
\noindent Technische Universit\"at M\"unchen,
\hfill {\tt carlos.amendola@tum.de} \\
North Carolina State University, \hfill {\tt bkholler@ncsu.edu} \\
North Carolina State University, \hfill {\tt smsulli2@ncsu.edu} \\
University of Austin, Texas, \hfill {\tt ntran@math.utexas.edu} \\

\end{document}